\documentclass[final,3p,times,sort]{elsarticle}

\usepackage{amssymb}
\usepackage{amsmath}
\usepackage{graphicx}
\usepackage{graphics}
\usepackage{algorithm}
\usepackage{algorithmic}
\usepackage{expl3}
\usepackage{enumitem}  
\usepackage{subfig}

\usepackage[usenames,dvipsnames,svgnames,table]{xcolor}
\usepackage[pagebackref=true,colorlinks,linkcolor=blue,citecolor=magenta]{hyperref}
\usepackage{tikz}
\tikzstyle{vertex}=[circle, draw, inner sep=0pt, minimum size=3pt]
\newcommand{\vertex}{\node[vertex]}
\newcommand{\comment}[1]{}

\newtheorem{theorem}{Theorem}
\newtheorem{corollary}[theorem]{Corollary}
\newtheorem{lemma}[theorem]{Lemma}
\newtheorem{proposition}[theorem]{Proposition}
\newtheorem{observation}[theorem]{Observation}
\newtheorem{conjecture}[theorem]{Conjecture}
\newtheorem{problem}[theorem]{Problem}
\newtheorem{claim}[theorem]{Claim}
\newtheorem{remark}[theorem]{Remark}

\newtheorem{preexample}{{\bf Example}}

\newtheorem{preproof}{{\bf Proof}}

\newenvironment{proof}[1]{\begin{preproof}{\rm
              #1}\hfill{$\blacksquare$}}{\end{preproof}}

\newtheorem{presproof}{{\bf Sketch of Proof}}

\def\S{\mathcal{S}}

\DeclareMathOperator\chain{\operatorname{chain}}
\DeclareMathOperator\prev{\operatorname{prev}}
\DeclareMathOperator\nex{\operatorname{next}}

\journal{Journal}

\begin{document}

\begin{frontmatter}



\title{Maximum Nullity and Forcing Number on Graphs\\ with Maximum Degree at most Three}
\author{Meysam Alishahi \corref{cor1}}\cortext[cor1]{Corresponding author.}
\ead{meysam$\_$alishahi@shahroodut.ac.ir}
\author{Elahe Rezaei Sani} \ead{e.rezaeisani24@gmail.com}
\author{Elahe Sharifi} \ead{e.sharifi1988@gmail.com}
\address{Faculty of Mathematical Sciences, Shahrood University of Technology, Shahrood, Iran.}
\begin{abstract}
A dynamic coloring of the vertices of a graph $G$ starts with an initial subset $F$ of colored vertices, with all remaining vertices being non-colored. At each time step, a colored vertex with exactly one non-colored neighbor forces this non-colored neighbor to be colored. The initial set $F$ is called a forcing set of $G$ if, by iteratively applying the forcing process, every vertex in $G$ becomes colored.
The forcing number of a graph $G$, denoted by $F(G)$, is the cardinality of a minimum forcing set of $G$.   
The maximum nullity of $G$, denoted by $M(G)$, is defined to be the largest possible nullity over all real symmetric matrices $A$ whose $a_{ij} \neq 0$ for $i \neq j$, whenever two vertices $u_{i}$ and $u_{j}$
of $G$ are adjacent.  
In this paper, we characterize all graphs $G$ of order $n$, maximum degree at most three, and $F(G)=3$. 
Also we classify these graphs with their maximum nullity. 
\end{abstract}

\end{frontmatter}
\section{Introduction}
Following the notation of \cite{dhmp,bbfhhsdhz}, for a graph $G$, we define the forcing process as follows:
Let $F\subseteq V (G)$ be an initial set of vertices. 
At  step zero, each vertex in $F$ is colored and any other vertex is non-colored. 
For $i\geq 1$, at step $i$, each colored vertex $v$ with exactly one non-colored neighbor forces its non-colored 
neighbor to be colored. At each step $i$, the set of newly get colored vertices is denoted by $V_i$ that in particular $V_0=F$. 
This process will continiue until there is no more possible change. At the end of the process, the 
set of colored vertices is called the \textit{derived set}. 
A \textit{forcing set} (zero forcing set) of $G$ is a set $F\subseteq V(G)$ of initially colored vertices if the corresponding derived 
set is equal to $V(G)$, i.e., $V(G)=\bigcup V_i$. For each vertex $v\in V(G)$, the step that $v$ gets colored is denoted by $k_F(v)$. 
Note that  $k_F(v)=i$ if and only if $v\in V_i$. 
The \textit{forcing number} of a graph $G$, originally known as the zero forcing number, denoted by $F(G)$, is 
the minimum possible cardinality for a forcing set of $G$ and any forcing set of $G$ of cardinality $F(G)$ is named an 
$F(G)$-set. It is simple to see that $F(G)=|V(G)|$ if and only if $G$ is a trivial graph i.e., a graph with no edge.   

Let $F$ be an $F(G)$-set. In the forcing process with initial vertex set $F$, 
any vertex forces at most one vertex  
while a vertex can be forced by different vertices.  Therefore, $V(G)$ can be partitioned to $|F|$ ordered 
sequences (chains) $R_1,\ldots,R_{|F|}$.
For each of these sequence, say $R_i=(v_{0},v_{1},\ldots,v_{k})$ where $i \in [|F|]$ , 
$v_0\in F$ and $v_j$ is forced by $v_{j-1}$ for each $j\in [k]$.  The set $\{R_1,\ldots,R_{|F|}\}$ is called a chain set with respect to $F$ for $G$, 
which is not  necessarily unique.  
For each $j\in\{0,\ldots,k-1\}$, we define $\nex_{R_i}(v_j)=v_{j+1}$ and 
for each $j\in[k]$, we set $\prev_{R_i}(v_j)=v_{j-1}$. 
Each sequence $R_i=(v_{0},v_{1},\ldots,v_{k})$ induces a path $P_i=v_{0} v_{1} \ldots v_{k}$ in $G$, i.e, $G[V(R_i)]$ is an induced path. 
To see that, note that if there exist two adjacent vertices $v_i,v_j \in V(R_i)$, where $j> i+1$, then when $v_i$ forces $v_{i+1}$, its neighbor  $v_j$ 
is non-colored neighbor which is impossible. A chain is called {\it trivial} if it contains only one vertex.

For a graph $G$, a \textit{total forcing set} is a forcing set of $G$ inducing a subgraph with no isolated vertex. 
For simplicity, we write TF-set instead of total forcing set.  
The \textit{total forcing number} of $G$, denoted by $F_{t}(G)$, is the cardinality of a minimum $TF$-set in $G$. Total forcing set was first introduced and studied by Davila~\cite{db} as a strengthening of forcing set which was originally introduced by  the AIM-minimum rank group~\cite{aim}. It is observed (for instance, see~\cite{db,dho}) that for a graph $G$ with no isolated vertex,   
\begin{equation}\label{ft}
F(G) \leq F_{t}(G) \leq 2F(G). 
\end{equation}

The forcing process and the forcing number were introduced in \cite{bg} and \cite{aim} to bound the minimum rank of a graph and hence its maximum nullity.
Since then, the forcing number has gained a considerable attention in graph theory and has been related to many graph theoretic parameters. In general, study of forcing number is challenging for many reasons. First, it is difficult to compute, as it is known to be NP-hard~\cite{d,td}. Further, many of the known bounds leave a wide gap for graphs in general. For example, the forcing number of a graph of order $n$ can be as low as $\delta(G)$ (see~\cite{bbfhhsdh}), and as high as $n\Delta(G)/(\Delta(G)+1)$ (see~\cite{acdp}).

Gentner and Rautenbach~\cite{grs} proved that for a graph $G$ with maximum degree at most $3$ and the
girth at least $5$, we have 
\begin{equation}
\label{eq:f}
\centering
F(G) \leq (\dfrac{n}{2})-\dfrac{n}{24 \log _{2}{n}+6}+2, 
\end{equation}
moreover, they introduced two graph $G_1$ and $G_2$ and proved $F(G) \leq n(\Delta -2)/(\Delta -1)$ for any connected graphs $G\not\in \{ K_{\Delta +1},K_{\Delta ,\Delta},K_{\Delta -1 ,\Delta},G_{1},G_{2}\}$. For more details and the definition of $G_1$ and $G_2$, see~\cite{grs}. 
They also conjectured that $F(G) \leq {n\over3}+2$ for any graph $G$ with $n$ vertices and $\Delta =3$. In~\cite{gms}, an infinite family of graphs $\lbrace G_{n} \rbrace $ with maximum degree $3$ was introduced such that the forcing number of $G_{n}$ is at least ${4\over 9}\vert V(G_{n}) \vert $, a counter example to the  Gentner-Rautenbach conjecture. Note that, at this point, the best upper bound for the forcing number of connected graphs with maximum degree three~\cite{acdp} is 
\begin{equation}
\label{eq:f1}
\centering
  F(G) \leq \frac{n}{2}+1.
\end{equation}
The equality can be achieved for graphs $K_{4} $ and $ K_{3,3} $. Davila and Henning~\cite{dhtz} studied forcing sets 
and total forcing sets in claw-free cubic graphs. They proved $ F(G)<n/2$, where $G$ is a connected, claw-free cubic graph with $n\geq10$ vertices.  
Akbari and Vatandoost~\cite{avp} gave a partial answer to the question
of determining all graphs $G$ with $M(G) = F(G)$ posed by AIM Minimum Rank-Special Graphs Work Group~\cite{aim}, 
where $M(G)$ is the maximum nullity of $G$ (for the definition, see Section~\ref{Sec1.1}).  
Indeed, they characterized all cubic graphs with forcing number $3$ and, as a corollary, 
they also showed that $M(G) = 3$ for any graph $G$ in this family. 
\begin{problem} {\rm\cite{aim}}\label{q1}  
 Determine all graphs $G$ for which $M(G) = F(G)$.
\end{problem}

In this paper, we characterize all graphs $G$ with maximum degree at most three and forcing number $3$. 
We will moreover investigate the maximum nullity of these graphs comparing to their forcing numbers. 
In this regard, we give a partial answer to Problem~\ref{q1} so that the family provided by Akbari and Vatandoost~\cite{avp} is included. 
Then we introduce a family of graphs with maximum degree at most three containing graphs $ G$ with $ M(G)=F(G)=k \leq 3 $ 
and we determine an upper bound for total forcing number of this family of graphs (see theorem~\ref{fmk}).

\subsection{ Definition and Notation}\label{Sec1.1}
For notation and terminology not presented here, we refer readers to~\cite{hhs}. 
We will use the notation $P_n$ and $C_n$ to denote the \textit{path} and the \textit{complete graph} on $n$ vertices, 
respectively. Also we use the standard notation $[k]=\{1,2,\ldots,k\}$. 
Let $G$ be a graph. For a path $P=x_0\cdots x_m$ in $G$,  
we define $x_iP=x_i\cdots x_m$ and
$Px_i=x_0\cdots x_i$ for each $i\in[m]$.


Let $S_{n}(\mathbb{R})$ be the set of all symmetric matrices of order n over the real numbers. For 
$A=\big(a_{ij}\big) \in S_{n}(\mathbb{R})$, the graph of $A$, denoted by $\mathcal{G}(A)$, is a graph with vertex 
set $\lbrace v_{1},\ldots,v_{n} \rbrace$ and edge set $ \left\{v_{i}v_{j} : a_{ij} \neq 0, 1 \leq i < j \leq n \right\}$. It should be noted that the diagonal of $A$ has no
role in the definition of $\mathcal{G}(A)$.
The set of symmetric matrices of graph $G$ is $S(G) =\{A \in S_{n}(\mathbb{R}) : \mathcal{G}(A) = G\}$. The \textit{minimum
rank of a graph $G$}, denoted by $mr(G)$, is the minimum possible rank for a matrix in $S(G)$ and,  similarly, the \textit{maximum nullity} of $G$, denoted by $M(G)$, is the maximum nullity of symmetric matrices in $S(G)$. Clearly, $mr(G) + M(G) = n$ and $M(G)\geq 1$ for any graph $G$.

\subsection{Graphs of $ k $-parallel paths}
Johnson et. al.~\cite{jls} defined the graph of $2$- parallel paths. A graph $G$, which is not a path, is said to be a 
\textit{graph of $2$-parallel paths} if there are two vertex disjoint induced paths covering all the vertices and $G$ 
can be drawn in the plane so that these two paths are parallel horizontal lines if we forgot their vertices, and 
moreover the edges (drawn as segments, not curves) with ends in different paths do not cross.  
Particularly, union of two disjoint paths is a graph of $2$-parallel paths.

As a generalization, in a natural way, we can define the \textit{graph of $k$-parallel paths} for any 
integer $k\geq 1$ as follows:  
Simply, a $1$-parallel path is just a path.
For an integer $k\geq 2$, a graph $G$, which is not a graph of $(k-1)$-parallel paths, is said to be a 
\textit{graph of $k$-parallel paths},  if there exist $k$ vertex disjoint induced 
paths covering all the vertices and $G$ can be drawn in the plane in a way that these paths are parallel horizontal lines if we forgot their vertices, and moreover 
the edges (drawn as segments, not curves) whose ends are in different paths do not cross each other. 
Such a drawing is called a {\it standard drawing} of $G$, for example see Figure~\ref{K5}. Note that a graph $G$ may have several standard drawings.
Here after, for a fixed standard drawing of a graph of $k$-parallel paths, the $k$ paths used in the definition are called the {\it parallel paths} 
with respect to this drawing and any edge with end-points in different paths is called a {\it segment}. 


It is clear that any graph of $k$-parallel paths has at least $k$ vertices. 
Also, It is known that every planar graph can be drawn in the plane in such a way that its edges are segment intersecting only at their endpoints, see~\cite{f}. Consequently, every planar graph of order $n$ is a graph of 
$k$-parallel paths for some $k\leq n$. 
Clearly, a non-planar graphs of order $n$ is not a graph of $n$-parallel paths. 
However, the following assertion indicates that the complete graphs of order $n\geq6$ are~not graphs of $k$-parallel paths for each $k \in \left[  n \right]$.
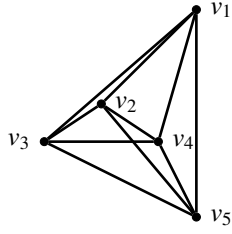
\begin{figure}
\centering
\begin{tikzpicture}
\vertex[fill] (a1) at (0,0) [label=left:$v_{3}$]{};
\vertex[fill] (a2) at (1.5,0) [label=right:$v_{4}$]{};
\vertex[fill] (a3) at (2,1.75) [label=right:$v_{1}$]{};
\vertex[fill] (a4) at (0.75,0.5) [label=right:$v_{2}$]{};
\vertex[fill] (a5) at (2,-1) [label=right:$v_{5}$]{};
\path [line width=1pt]
(a1) edge (a2)
(a1) edge (a3)
(a1) edge (a4)
(a1) edge (a5)
(a2) edge (a3)
(a2) edge (a4)
(a2) edge (a5)
(a3) edge (a4)
(a3) edge (a5)
(a4) edge (a5)
;
\end{tikzpicture}
\caption{\small A standard drawing of $ K_{5} $ as a graph of $4$-parallel paths with parallel paths $ P^{1}:v_{1}, P^{2}:v_{2}, P^{3}:v_{3}v_{4}, P^{4}:v_{5}$. }\label{K5}
\end{figure}

\begin{observation}\label{two}
For $n \geq 6 $,  $K_n$ is~not a graph of $k$-parallel paths for each $k \in [n]$.
\end{observation}
\begin{proof}{
For a contradiction, suppose that $K_n$ is a graph of $k$-parallel paths where $k\in[n]$.
Consider a standard drawing of $K_n$. 
Since $K_n$ is a complete graph and the parallel paths are induced paths, then each of its parallel paths is of order at most two. 
Furthermore, since $K_n$ is not a planar graph, we know $k\leq n-1 $. Consequently, there must be some parallel path isomorphic to $P_{2}$. 
Clearly, there is only one such a $P_{2}$ path, otherwise, the segments between two $P_2$ paths intersect each others which is~not possible. 
Therefore, there exist $n-2$ parallel paths $P^1,\ldots,P^{n-2}$ isomorphic to $P_1$ and exactly one parallel path $P^{n-1}$ isomorphic to $P_{2}$. 
Let $V$ be all the vertices in $K_n$ but a vertex from $P^{n-1}$. In view of the standard drawing of $K_n$, the induced subgraph by $V$ is a planar graph
isomorphic to $K_{n-1}$ which is~not possible since $n\geq 6$.
}\end{proof}

\section{Main results}\label{mainresults}

Akbari et. al.~\cite{avp} characterized all cubic graphs with forcing number $3$ and proved that the maximum nullity of these graphs is $3$ as well.   
In the next two results, as a generalization of their result, we characterize all the graphs with maximum degree at most three and $F(G)=3$. 
\begin{theorem}\label{iffkparallel}
For a graph $ G $ with $ \Delta(G) \leq 3 $, $ F(G)=3 $ if and only if $ G $ is a graph of $3$-parallel paths.
In particular, the left-most vertices of the parallel paths in any standard drawing of $G$ form an $F(G)$-set. 
\end{theorem}
As a consequence of Theorem \ref{iffkparallel}, in the following theorem, we give a partial answer to Problem~\ref{q1}:  
We provide a characterization of graphs with maximum degree at most three whose forcing number and maximum nullity are at most three.  

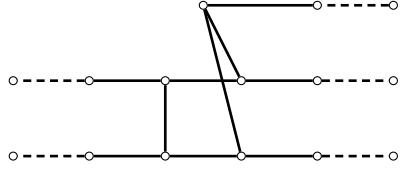
\begin{figure}
\centering
\begin{tikzpicture}
\vertex[] (a1) at (0.5,1) []{};
\vertex[] (a2) at (2,1) []{};
\vertex[] (a3) at (3,1) []{};
\vertex[] (a4) at (1,0) []{};
\vertex[] (a5) at (2,0) []{};
\vertex[] (a6) at (3,0) []{};
\vertex[] (a7) at (1,-1) []{};
\vertex[] (a8) at (2,-1) []{};
\vertex[] (a9) at (3,-1) []{};
\vertex[] (a10) at (0,0) []{};
\vertex[] (a11) at (-1,0) []{};
\vertex[] (a12) at (-2,0) []{};
\vertex[] (a13) at (0,-1) []{};
\vertex[] (a14) at (-1,-1) []{};
\vertex[] (a15) at (-2,-1) []{};
\path [line width=1pt]
(a1) edge (a2)
(a1) edge (a7)
(a1) edge (a4)
(a4) edge (a5)
(a10) edge (a4)
(a7) edge (a8)
(a7) edge (a13)
(a10) edge (a11)
(a10) edge (a13)
(a13) edge (a14)
;
\draw[densely dashed ,line width=1pt] (a5) -- (a6);
\draw[densely dashed ,line width=1pt] (a3) -- (a2);
\draw[densely dashed ,line width=1pt] (a8) -- (a9);
\draw[densely dashed ,line width=1pt] (a14) -- (a15);
\draw[densely dashed ,line width=1pt] (a11) -- (a12);
\end{tikzpicture}
\caption{\small Graphs with  $F(G)=M(G)+1=3$}\label{F8}
\end{figure}

\begin{theorem}\label{fmk} 
For a graph $G$ with $\Delta(G) \leq 3$, the following assertions hold. 
\begin{enumerate}
\item $F(G)=M(G)=1$ if and only if $G$ is a path. 
\item $F(G)=M(G)=2$ if and only if $G$ is a graph of $2$-parallel paths. 
\item $F(G)=M(G)+1=3$ if and only if $G$ is a graph of type defined in Figure~\ref{F8}. 
\item $F(G)=M(G)=3$ if and only if $G$ is a graph of $3$-parallel paths and not of type defined in Figure~\ref{F8}.
\item There is no graph with $F(G)=M(G)+1=2$ or $F(G)=M(G)+2=3$.
\end{enumerate}
\end{theorem}
It is worth noting that since $M(G)\leq F(G)$, Theorem~\ref{fmk} includes all possible case for $F(G)\leq3$. 

Intriguingly, one may be interested in total forcing number of $k$-parallel paths.  
In Section~ \ref{pt4}, we prove the next result asserting that the  total forcing number of a graph of $k$-parallel paths is at most twice the number of 
parallel paths which is a sharp bound. 
\begin{corollary}\label{ftn} 
For a graph $ G $ of $ k $-parallel paths and without isolated vertices, $F_{t}(G) \leq 2k$ and this bound is sharp.
\end{corollary}

We close this section by the following conjecture which is supported by Theorems~\ref{iffkparallel} and \ref{ftwo} and also by 
 Observation~\ref{p1} and Remark~\ref{rem1}. 
\begin{conjecture}\label{conj}
For a graph $G$ with $\Delta(G) \leq 3$, $F(G) = k$ for some $k\in[ n ]$ if and only if $G$ is a graph of $k$-parallel paths.
\end{conjecture}

\section{Known Results and Preliminary Lemmas}
This section is devoted to review some known results and to prove some key lemmas being used for the proofs of main results. We start with the following obvious observation.
\begin{observation}\label{p1}{\rm \cite{r}}
For a graph $G$,  $F(G)=1$ if and only if $G$ is isomorphic to a path. 
\end{observation}
A stronger similar result was proved in~\cite{aim}.
\begin{theorem}\label{M=0}{\rm \cite{aim}}
For a graph $G$,  $M(G)=1$ if and only if $G$ is isomorphic to a path. 
\end{theorem} 

As stated before, the forcing number has been defined as a tool for studying the maximum nullity of graphs. 
The next theorem states that the maximum nullity of a graph $G$ does not exceed $F(G)$.
\begin{proposition}\label{mg}{\rm \cite{aim}} 
For any graph $G$, $M(G) \leq F(G)$. 
\end{proposition}
Since, for computing $M(G)$, we need to evaluate an infinite number of matrices and for computing the 
forcing number we may only check a finite number of sets, researchers interested 
in knowing the maximum nullity of a graph will find Proposition~\ref{mg} very useful. 
The following theorem introduced some graphs satisfying equality in Proposition~\ref{mg} and thus some graphs for 
which the maximum nullity could be computed.

It is known that the graphs with maximum nullity one are exactly the paths~\cite{aim}.
Characterizing the graphs with maximum nullity two,  
Johnson et. al.~\cite{jls} introduced a family of graphs $\mathcal{F}$ containing six types of graphs
(for more detail, see Table B1 in~\cite{jls})  and proved the next theorem.  
\begin{theorem}\label{mtwo}{\rm\cite{jls}} For a graph $G$, $M(G) = 2$ if and only 
if $G$ is a graph of $2$-parallel paths or $G$ is in $\mathcal{F}$.
\end{theorem}
By Proposition~\ref{mg}, a graph with maximum nullity two has the forcing number at least two.  
Using Theorem~\ref{mtwo}, next result by Row~\cite{r} characterizes the graphs with forcing number and maximum nullity two. 
\begin{theorem}\label{ftwo}\cite{r} 
For a graph $G$, $F(G) = 2$ if and only if $G$ is a 
 graph of $2$-parallel paths. 
\end{theorem} 
\begin{theorem}\label{mf}\cite{aim} 
For each of the following families of graphs, $F(G) =M(G)$; 
\begin{itemize}

\item[ 1)] any graph $G$ with $\vert G \vert \leq 6$, 

\item[2)] $K_n, P_n, C_n$,  and

\item[3)] any tree $T$.
\end{itemize}
\end{theorem}

Let $G$ be a graph and $F$ be an $F(G)$-set. Consider a chain set $\S =\{R_{1},\ldots,R_{|F|}\}$ with respect to $F$. 
Note that each of these chains is an induced path in $G$.  
We are interested to find some necessary condition for $\S$ to have a standard drawing of $G$ whose 
parallel paths are exactly these chains. 
To do it, in this paper we consider these chains (induced paths) as horizontal paths such that the set of  left-most vertices of these paths is $F$. Also, we draw the edges with ends in different paths as segments (not curves) with the least possible interruption between them
(no three edges can have a common point but in the vertices). 
Note that in this case if $G$ has zero interruption, then it is a graph of $k$-parallel paths where $k\in[|F|]$. 
In this drawing, any edge with end points in different paths is called a {\it segment}. For each chain $R_i$, it is clear that the forcing process 
induces an ordering to the vertices of $R_i$ increasing from left to right. 
In other words, $u <_{R_i} v$ means that $u$ and $v$ are in the chain $R_i$ and $u$ precedes $v$ in this chain. 
By $\chain_{\S}(v) $, we denote the forcing chain in $\S$ containing $v$. 
Furthermore, $\nex_{R_i}(v)$ is the neighbor of $v$ in the chain $R_i$ such that $v<_{R_i} \nex_{R_i}(v)$, 
if there exist; $\prev_{R_i}(v)$ is defined analogously. 


In our approach to prove the main results, we crucially deal with different possible cases of the number of trivial and 
non-trivial chains corresponding to an $F(G)$-set of a graph $G$. 
In this regard, we first observe the following simple assertion.  

\begin{observation}\label{singleton}
Let $G$ be a non-trivial graph with $F(G) = k$ and $ F $ an $ F(G) $-set. Every chain set $\S$ corresponding to $F$ has at most $k-1$ trivial chains.
\end{observation}
\begin{proof}{
For a contradiction, suppose that $\S$ has $k$ trivial chains and consequently, $\vert V(G) \vert =k $. But, 
we know $k=\vert F \vert \leq \vert V(G) \vert -1 =k-1 $, a contradiction.
}\end{proof}

Although, the following lemma has a straight proof, it plays a key role throughout the proofs.

\begin{lemma}\label{cross1} 
Let $G$ be a graph, $F$ an $F(G)$-set, and $\S$ a chain set corresponding to $F$. 
Let $x$ and $y$ be two vertices such that $x,y \in R \in \S$ and $x<_{R}y$. Then $k_F(z)<k_F(y)$ for any $z\in N(x)\setminus V(R)$. 
\end{lemma}
\begin{proof}{
When $x$ is forcing $\nex_{R}(x)$, all the neighbors of $x$ but $\nex_{R}(x)$ are colored. 
Therefore, $z$ is colored before  $\nex_{R}(x)$ and consequently before $y$. 
This implies that $k_F(z)< k_F(y)$.
}\end{proof}
The next two lemmas are immediate consequences of this lemma. 
\begin{lemma}\label{cross} 
Let $G$ be a graph, $F$ an $F(G)$-set, and $\S$ a chain set corresponding to $F$. 
For any edge $xy$ with endpoints $x$ and $y$ in different chains $R_1,R_2\in \S$ respectively, there is no edge $x'y'$ 
such that $x'\in R_1, y'\in R_2$, and $x<_{R_1}x'$ and $y'<_{R_2}y$. 
\end{lemma}
\begin{lemma}\label{crosss} 
Let $G$ be a graph, $F$ an $F(G)$-set, and $\S$ a chain set corresponding to $F$. 
If there are three different chains $R_1,R_2,R_3\in \S$ and two edges $ab$ and $cd$ such that $a\in R_1, b,c\in R_2$, $d\in R_3$ and $c<_{R_2}b$, then there is no edge $xy$ such that $a<_{R_1}x$ and $y<_{R_3}d$.
\end{lemma}


Let $F$ be an $F(G)$-set of a graph $G$ and $\S$ a chain set corresponding to it. 
A vertex $v$ is called {\it bad with respect to $\S$} if there are two non-trivial chains $R_1,R_2\in \S$ such that $v\in R_1$ 
has two non-consecutive neighbors in $R_2$. When the chain set $\S$ is clear from the context, the vertex $v$ is called a bad vertex. 
Note that only vertices of non-trivial chains can be bad.
\begin{lemma}\label{initialbad}
Let $G$ be a graph with $\Delta(G)\leq 3$, $F$ an $F(G)$-set, and $\S$ a chain set corresponding to $F$. 
Every bad vertex with respect to $\S$ is an initial vertex of a chain in $\S$. 
\end{lemma}
\begin{proof}{
Let $x\in R_1$ be a bad vertex with respect to $\S$, where $R_1$ is a chain in $\S$. 
In view of the definition of bad vertex, $R_1$ is not trivial and, 
consequently, has at least two vertices.  
Note that, since $\Delta(G)\leq 3$ and $R_1$ is an induced path, $x$ must be an end-vertex of $R_1$. 
Let $a,b\in R_2$ be the two neighbors of $x$, where $R_2\neq R_1$ is a chain in $\S$.  
For simplicity, assume that $a<_{R_2}b$, i.e., $a$ is the first  neighbor of $x$ according to the ordering on $R_2$. 
Since $a$ and $b$ are~not consecutive, $a<_{R_2}\nex_{R_2}(a)<_{R_2}b$. 
For a contradiction, suppose that $x$ is the last vertex of $R_1$. 
Note that when $a$ is forcing $\nex_{R_2}(a)$, the vertex $x$ must be colored and consequently,
since $x$ is the last vertex of $R_1$, the vertex $b$ can be forced at this step by $x$. 
This implies that $\nex_{R_2}(a)$ and $b$ can~not be in the same chain, a contradiction. }\end{proof}

\begin{lemma}\label{nobadlemma}
For a graph $G$ with $\Delta(G)\leq 3$  and $F(G)=3$, there are an $F(G)$-set $F$ 
and a chain set $\S$ corresponding to it with no bad vertex. 
\end{lemma}
\begin{proof} {
Let $F=\{x,y,z\}$ be an $F(G)$-set and $\S=\{R_1,R_2,R_3\}$ be a 
 chain set corresponding to it having the minimum possible of bad vertices 
with respect to all possible choices of $F$ and $\S$, where $x,y$, and $z$ are respectively 
the initial vertices of $R_1,R_2$, and $R_3$. In what follows, completing the proof, 
we will show that $G$ has no bad vertex with respect to $\S$. In view of Lemma~\ref{initialbad}, 
the only possible bad vertices are $x,y$, and $z$. 
Suppose that $x$ is bad with respect to $\S$. 
For simplicity of notation, assume that $x$ has two non-consecutive neighbors $a$ and $b$ on $R_2$, where $a<_{R_2}b$ (see Figure~\ref{f2}). 
Set $x'=\nex_{R_2}(a)$ and define $F'=(F\setminus\{x\})\cup\{x'\}$ and 
$\S'=(\S\setminus\{R_1,R_2\})\cup\{R'_1,R'_2\}$, where   $R'_1=x'R_2$ and 
$R'_2=R_2axR_1$. One can check that $\S'$ is  a chain set corresponding to $F'$. 
Note that, in view of Lemma~\ref{initialbad}, the possible bad vertices with respect to $\S'$ are $x',y$, and $z$. 
\begin{claim}
The vertex $z$ is~not bad with respect to $\S$ and is bad with respect to $\S'$. 
\end{claim}
\noindent{\bf Proof of Claim.}
In view of how we chose $F$ and $\S$, the number of bad vertices in $G$ with respect to $\S'$ can~not be less than 
the number of those in $\S$. 
Therefore, to fulfill the claim, it suffices to prove the two following items. 
\begin{itemize}
\item[{\bf I)}] {\bf The vertex $x'$ is~not bad with respect to $\S'$.} 
For a contradiction, suppose that $x'$ is bad with respect to $\S'$. 
Since $x'$ has exactly one neighbor on $R'_1$ and $a\in R'_2$ is adjacent to $x'$,
the vertex $x'$ has two non-consecutive neighbors in $R'_2$, say $a$ and $c$, where $x<_{R'_2}c$.
But, in view of  Lemma~\ref{cross}, this is~not possible since the $x<_{R_1}c$, $x'<_{R_2}b$ and $xb, x'c\in E(G)$.
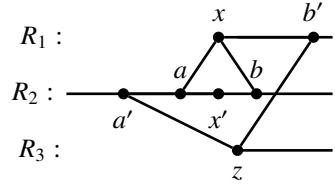
\begin{figure}
\centering
\begin{tikzpicture} [scale=1,very thick] 
\draw[line width=1pt] (0.75,0.75) -- (2,0.75);
\draw[line width=1pt] (-1.5,0) -- (2,0);
\draw[line width=1pt] (0.75,-0.75) -- (2,-0.75);
\vertex[fill] (a1) at (0,0) [label=above:$a$]{};
\vertex[fill] (a2) at (0.5,0.75) [label=above:$x$][label=left:$R_{1}:\qquad\qquad\quad\:$]{};
\vertex[fill] (a3) at (0.75,-0.75) [label=below:$z$][label=left:$R_{3}:\qquad\qquad\qquad$]{};
\vertex[fill] (a4) at (1.75,0.75) [label=above:$b^{\prime}$]{};
\vertex[fill] (a5) at (0.5,0) [label=below:$x'$]{};
\vertex[fill] (a6) at (1,0) [label=above:$b$]{};
\vertex[fill] (a7) at (-0.75,0) [label=below:$a'$][label=left:$R_{2}:\quad\quad$]{};
\path [line width=1pt]
(a1) edge (a2)
(a1) edge (a5)
(a1) edge (a7)
(a2) edge (a4)
(a2) edge (a6)
(a3) edge (a7)
(a3) edge (a4);
\end{tikzpicture}
\caption{\small A drawing of $ \S $ in Lemma~\ref{nobadlemma} with assumption that $ x $ of chain $ R_{1} $ is adjacent to two non-consecutive vertices $ a $ and $ b $ of chain $ R_{2} $; and $z$ of chain $ R_{3} $ is adjacent to the vertices $ a' $ and $ b' $ of the chains $ R_{1} $ and $ R_{2} $ respectively.}\label{f2}
\end{figure}

\item[{\bf II)}] {\bf If $y$ is bad with respect to $\S'$, then it is also bad with respect to $\S$.} 
Since $R_2$, $R'_1$, and $R'_2$ are induced paths in $G$ and $y$ is the initial vertex of $R_2$ and $R'_2$, 
in view of the definition of $R'_1$ and $R'_2$, 
the vertex $y$ has at most one neighbor in $R'_1$. 
Hence, if $y$ is bad with respect to $\S'$, then it must have two non-consecutive neighbors 
in $R_3$ which implies that it is bad with respect to $\S$ as well. \hfill{$\Box$}\\
\end{itemize}
Hence, in view of the aforementioned claim, $z$ must have two non-consecutive neighbors on $R'_2$, one located 
in $R_2a$, say $a'$, and the other located in $\nex_{R_1}(x)R_1$, say $b'$ (see Figure~\ref{f2}). 
Note that, in view of Lemma~\ref{crosss}, there is no edge $uv$ such that $u<_{R_2}b$ and $z<_{R_3}v$.
In particular, $x'$ has no neighbor in $R_3$. 
Now, we define $F''=(F'\setminus\{z\})\cup\{z''\}$ and $\S''=(\S'\setminus\{R_2',R_3\})\cup\{R''_2,R''_3\}$, 
where $z''=\nex_{R'_2}(a')$,
$R''_2=R'_2a'zR_3$, and $R''_3=z''R'_2$. One can check that $\S''$ is a chain set corresponding to $F''$.
\begin{claim}
The graph $G$ has no bad vertex with respect to $\S''$. 
\end{claim}
Note that $F''=\{x',y,z''\}$ where $x',y,$ and $z''$ are respectively the initial vertices of $R'_1,R''_2$ and $R''_3$. 
In view of Lemma~\ref{initialbad}, the possible bad vertices with respect to $\S''$ are $x',y$, and $z''$. 
\begin{itemize}
\item[{\bf I)}] {\bf The vertex $x'$ is~not bad with respect to $\S''$.} 
If $x'$ is bad with respect to $\S''$, then 
since $x'$ is~not bad with respect to $\S'$, the path $R'_1$ is an induced non-trivial path, and $\Delta(G)\leq 3$, 
the vertex $x'$ must have two non-consecutive neighbors on $R''_2$. 
This implies that $y=a=a'$ and $z<_{R_3}d$, 
where $d\in R''_2$ is the other neighbor of $x'$. But we already knew that $x'$ can~not have a neighbor succeeding $z$ on $R_3$. 

\item[{\bf II)}] {\bf The vertex $z''$ is~not bad with respect to $\S''$.} 
If $z''$ is bad with respect to $\S''$, then it is easy to see that $a\neq a'$ and $z''$ has a neighobr in $R_3$ succeeding $z$.
This is~not possible since  $z''<_{R_2}b$ and, in view of Lemma~\ref{crosss}, we know that $z''$ cannot have a neighbor succeeding $z$ on $R_3$

\item[{\bf III)}] {\bf The vertex $y$ is~not bad with respect to $\S''$.} 
Since $R_2$ and $R'_2$ are induced paths, $y$ has at most one neighbor in each of $R''_3$ and $R'_1$. \hfill{$\Box$}\\
\end{itemize}
 
Therefore, the bad vertices with  respect to $\S''$ are less than of the bad vertices with respect to $\S$ which is impossible. 
}\end{proof}

\begin{lemma}\label{cons} 
Let $G$ be a graph with $ \Delta (G) \leqslant 3$ and $F(G)=3$. There is an $F(G)$-set $F$ 
and a chain set $\S$ corresponding to $F$ such that if a vertex $x$ of a 
non-trivial chain $R_1\in \S$ is adjacent to two vertices $y$ and $z$ of another chain $R_2\in \S$, 
then $x$ is either the initial or the end vertex of $R_1$ and $y,z$ are consecutive in $R_2$. 
\end{lemma}
\begin{proof}{
Let $F$ be an $F(G)$-set and $\S$ be a chain set corresponding to it 
whose existences are ensured by Lemma~\ref{nobadlemma}.
We claim that $F$ and $S$ satisfy the assertion of the lemma. To prove it, suppose that 
$x$ is a vertex of a non-trivial chain $R_1$ which is adjacent to two vertices $y$ and $z$ of another chain $R_2$. Since $\Delta(G)\leq 3$, the vertex $x$ is either the initial or the end vertex of $R_1$. To complete the proof, we need to show that $y$ and $z$ are consecutive in $R_2$. For a contradiction, assume that this is~not the case. Consequently, $x$ must be a bad vertex with respect to $\S$, which is~not possible.
}\end{proof}

Let $F$ be an $F(G)$-set of a graph $G$ and $\S$ a chain set corresponding to it. 
A vertex $x$ is called {\it unfavorite} with respect to $\S$ if there are three non-trivial chains $R_1,R_2,R_3 \in \S$ such that $x \in R_1$ has two neighbors 
$a\in R_2$ and $b\in R_3$ and there is a segment $cd$ such that 
$a<_{R_2} c$, and $d<_{R_3} b$ (see Figure~\ref{f3}).

\begin{lemma}\label{unfavorite}
Let $G$ be a graph with $\Delta(G)\leq 3$, $F$ an $F(G)$-set, and $\S$ a chain set corresponding to $F$. 
Everey unfavorite vertex with respect to $\S$ is an initial vertex of a chain in $\S$.  
\end{lemma}
\begin{proof}{
Let $x$ be unfavorite with respect to $\S$. For simplicity of notation, assume that $x \in R_1$ has two neighbors 
$y\in R_2$ and $z\in R_3$ where $R_1$,$R_2$ and $R_3$ are distinct non-trivial chains and there is a segment $ab$ such that 
$y<_{R_2} a$, and $b<_{R_3} z$. Since $\Delta(G)\leq 3$ and $x$ has two neighbors on $R_2$ and $R_3$, 
the vertex $x$ msut be either the initial or the end vertex of $R_1$. For a contradiction, assume that $x$ is the end vertex of $R_1$. 
It is clear that $k_F(z)\leq \max\{k_F(x),k_F(y)\}+1$.  Moreover, in view of Lemma~\ref{cross1}, we have 
$k_F(a)<k_F(z)$ and  $k_F(x)<k_F(a)$ which implies 
$$k_F(y)<k_F(a)<k_F(z) \qquad\mbox{and}\qquad k_F(x)<k_F(a)<k_F(z).$$
Accordingly, $k_F(z)\geq\max\{k_F(x),k_F(y)\}+2$, a contradiction.
}\end{proof}

\begin{lemma}\label{unless} 
Let $G$ be a graph with $ \Delta (G) \leq 3 $ and $ F(G)=3$. There are an $F(G)$-set $F$ and a 
chain set $\S$ corresponding to $F$ satisfying Lemma \ref{cons} containing no unfavorite vertex. 
\end{lemma}

\begin{proof}{ 
Let $F=\{x,y,z\}$ be an $F(G)$-set and $\S=\{R_1,R_2,R_3\}$ a 
chain set corresponding to it whose existences are ensured by Lemma~\ref{cons} 
having the minimum possible number of unfavorite vertices, where $x,y$, and $z$ are respectively 
the initial vertices of $R_1,R_2$, and $R_3$. In what follows, completing the proof, we will show that $G$ has no unfavorite vertex with respect to $\S$. In view of Lemma~\ref{unfavorite}, the only possible unfavorite vertices are $x,y$ and $z$. For a contradiction suppose that $x$ is unfavorite. For simplicity of notation, assume that $x$ has two neighbors $a\in R_2$ and $b\in R_3$ and there is a segment $cd$ such that $a<_{R_2}c$ and $d<_{R_3}b$ (see Figure~\ref{f3}). In view of Lemma~\ref{cross1}, $k_F(c)<k_F(b)$ which implies that $k_F(a)<k_F(c)<k_F(b)$. Now, set $x'=\nex_{R_2}(a)$ and define $F'=(F\setminus \{x\})\cup \{x'\}$ and 
$\S'=(\S\setminus\{R_1,R_2\})\cup\{R'_1,R'_2\}$, where $R'_1=x'R_2$ and $R'_2=R_2axR_1$. 
One can check that $\S'$  is a chain set corresponding to $F'$.
Note that $F'=\{x',y,z\}$ where $x',y,$ and $z$ are respectively the initial vertices of $R'_1,R'_2,R_3$. 

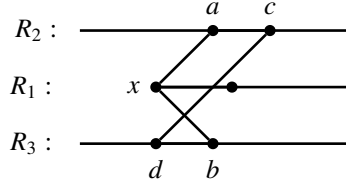
\begin{figure}
\centering
\begin{tikzpicture} [scale=1,very thick] 
\draw[line width=1pt] (-1,0.75) -- (2.5,0.75){};
\draw[line width=1pt] (0,0) -- (2.5,0){};
\draw[line width=1pt] (-1,-0.75) -- (2.5,-0.75){};
\vertex[fill] (a1) at (0,0) [label=left:$x$][label=left:$R_{1}:\qquad\;\quad$]{};
\vertex[fill] (a2) at (0.75,0.75) [label=above:$a$][label=left:$R_{2}:\qquad\qquad\quad\:$]{};
\vertex[fill] (a3) at (0.75,-0.75) [label=below:$b$]{};
\vertex[fill] (a6) at (1,0) []{};
\vertex[fill] (a4) at (1.5,0.75) [label=above:$c$]{};
\vertex[fill] (a5) at (0,-0.75) [label=below:$d$][label=left:$R_{3}:\qquad\;\quad$]{};
\path [line width=1pt]
(a1) edge (a2)
(a1) edge (a3)
(a1) edge (a6)
(a2) edge (a4)
(a3) edge (a5)
(a4) edge (a5)
;
\end{tikzpicture}
\caption{A drawing of $ \S $ in Lemma~\ref{unless}, with assumption that $ x $ of $ R_{1} $ is adjacent to the vertices $ a $ and $ b $ of $ R_{2} $ and $ R_{3} $ respectively, such that $c$ of $ R_{2} $, where $ c >_{R_2}a $, is adjacent to the vertex $ d $ of $ R_{3} $, where $d <_{R_3} b.$}\label{f3}
\end{figure}


\begin{claim}
$F'$ and $\S'$ satisfy Lemma~\ref{cons}. 
\end{claim}
\noindent{\bf Proof of Claim.}
In view of Lemma~\ref{initialbad} , to prove the assertion we need to show that $x',y$, and $z$ are not bad vertices 
with respect to $\S'$
\begin{itemize}

\item[{\bf I)}] {\bf The vertex $x'$ is~not bad with respect to $\S'$.} 
If $x'$ is bad with respect to $\S'$ then, by definition of bad vertices, $R'_1$ is not trivial and also $x' \neq c$ (otherwise $x'=c$ has one neighbor in each chain). 
Therefore, since $x'$ is adjacent to $a\in R'_2$ and $R'_2$ is an induced path, $x'$ must have two non-consecutive neighbors
$a$ and $e$ in $R_2'$ such that $e\in R_1$ and $x<_{R_1} e$. On the other hand, by Lemma~\ref{crosss}, $x'$ has no neighbor succeeding $x$ on $R_1$, a contradiction. 

\item[{\bf II)}] {\bf The vertex $z$ is~not bad with respect to $\S'$.} 
If $z$ is bad with respect to $\S'$, then, since $z$ is~not bad with respect to $\S$, 
it is easy to see that $z$ must have a neighobr in $R_1$ succeeding $x$.
By the Lemma~\ref{cross}, this is~not possible since  $xb\in E(G)$. 

\item[{\bf III)}] {\bf The vertex $y$ is~not bad with respect to $\S'$.} 
If $y$ is bad with respect to $\S'$, then since $R_2$ and $R'_2$ are both induced paths, $y$ must have two non-consecutive neighbors in $R_3$ implying that $y$ is bad with respect to $\S$ as well, a contradiction. \hfill{$\Box$}\\
\end{itemize}
 
By the next claim,  we will prove that $G$ has no unfavorite vertex with respect to $\S'$. 

\begin{claim}
The graph $G$ has no unfavorite vertex with respect to $\S'$. 
\end{claim}
\noindent{\bf Proof of Claim.}
In view of Lemma~\ref{unfavorite} , to prove the assertion we need to show that $x',y$, and $z$ are not unfavorite vertices respect to $\S'$
\begin{itemize}

\item[{\bf I)}] {\bf The vertex $x'$ is~not unfavorite with respect to $\S'$.} 
By Lemma~\ref{cross}, $x'$ cannot have any neighbor succeeding $d$. So, if $x'$ is unfavorite with respect to $\S'$, then 
it has two neighbors $a$ and $a'$, where $a'\in R_3$ and $a'\leq_{R_3} d$. 
Since $x'$ is unfavorite with respect to $S'$, one can see that there must be a segment $uv$ such that $u<_{R_1} a$ and $a'<_{R_3}v$ or a segment $u'v'$ such that $x<_{R_3}u'$ and $v'<_{R_3}a'$ which is~not possible due to Lemma~\ref{cross} (using this lemma for original $S$). 

\item[{\bf II)}] {\bf The vertex $z$ is~not unfavorite with respect to $\S'$.} 
In view of Lemma~\ref{cross}, $z$ has no neighbor in $R_1$. Therefore, 
if $z$ is unfavorite with respect to $\S'$, then it is easy to see that $z$ has two non-consecutive neighbors in $R_2$ implying that $z$ is bad with respect to $\S$, a contradiction to the assumption that $F$ and $\S$ satisfy Lemma~\ref{cons}. 

\item[{\bf III)}] {\bf The vertex $y$ is~not unfavorite with respect to $\S'$.} 
Since $R_2$ is an induced path, if $y$ is unfavorite, then clearly $y=a$. 
Otherwise, if $y\neq a$, then $y$ has no neighbor on $R'_1$ and consequently, is not unfavorite. 
Let the two neighbors of $y$ be $x'\in R'_1$ and $w\in R_3$. Now, in view of Lemma~\ref{cross}, one can see that no vertex 
in $R'_1$ is adjacent to a vertex of $R_3$ preceding $w$, a contradiction to the assumption that $y$ is unfavorite. \hfill{$\Box$}\\
\end{itemize}
Therefore, the number of unfavorite vertices with  respect to $\S'$ is less than the number of those with respect to $\S$ which is impossible.  
}\end{proof}
\section{A sufficient condition for being of $3$-parallel paths}
Let $G$ be a graph with $\Delta(G)\leq 3$ whose vertex set can be partitioned into three induced vertex disjoint paths $P^{1},P^{2},P^{3}$ 
satisfying the following properties: 
\begin{enumerate}
\item Each of $P^{1}$ and $P^{2}$ has at least two vertices
\item If $P^{3}$ is a singleton, then its degee is at most two.
\item For any edge $xy$ with endpoints $x$ and $y$ in different chains $P^{i}$ and $P^{j}$ respectively, there is no edge $x'y'$ 
such that $x'\in P^{i}, y'\in P^{j}$, and $x<_{P^{i}}x'$ and $y'<_{P^{j}}y$,
\item For $\{i,j,k\}=[3]$, 
if there are two edges $ab$ and $cd$ such that $a\in P^{i}, b,c\in P^{j}$, $d\in P^{k}$, and $c<_{P^{j}}b$, then there is no edge $xy$ such that $a<_{P^{i}}x$ and $y<_{P^{k}}d$.
\item For each $i\neq j\in[3]$, there is no vertex $v\in V(P^{i})$ having two non-consecutive neighbors in $P^{j}$. 
\item If $\{i,j,k\}=[3]$, then for each vertex $x \in P^{i}$ having two neighbors 
$a\in P^{j}$ and $b\in P^{k}$, there is no edge $cd$ such that 
$a<_{P^{j}} c$, and $d<_{P^{k}} b$. 
\end{enumerate} 

Since $\Delta(G)\leq 3$, by Property~5, if a vertex of $P^{1}$ (resp. $P^{2}$) has two neighbors in $P^{2}$ (resp. $P^{1}$), then 
this vertex is either the initial or the end vertex of $P^{1}$ (resp. $P^{2}$) and these neighbors are 
consecutive on $P^{2}$ (resp. $P^{1}$). Moreover, by the third property, at most one of the initial (resp. end) vertices of $P^{1}$ and $P^{2}$ can have two consecutive neighbors in another chain. 
If a vertex $u$ of $P^{1}$ (resp. $P^{2}$) has  two neighbors $v_1,v_2$ in $P^{2}$ (resp. $P^{1}$), then we glue these two consecutive neighbors together to have a thick vertex $\mathbf{v}$ and a thick edge between $u$ and $\mathbf{v}$. 
Now, we draw the two chains $P^{1}$ and $P^{2}$ as two parallel horizontal lines, if we forget their vertices, 
where $P^{1}$ is located above $P^{2}$ and the edges between them are vertical segments. 
Note that, in view of Property~3, this drawing has no interruption. 
We call this drawing a {\it ladder drawing} of  $P^{1}$ and $P^{2}$, see Figures~\ref{f4} and~\ref{f5}. 

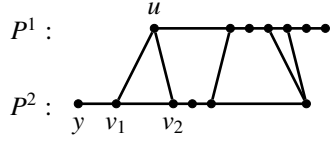
\begin{figure}
\centering
\begin{tikzpicture}
\draw[line width=1pt] (0,0) -- (2.25,0);
\draw[line width=1pt] (-1,-1) -- (2,-1);
\vertex[fill] (a1) at (0,0) [label=above:$u$][label=left:$P^{1}:\quad\quad\quad\;$]{};
\vertex[fill] (a2) at (-1,-1) [label=below:$y$][label=left:$P^{2}:\:\:$]{};
\vertex[fill] (a3) at (-0.5,-1) [label=below:$v_1$][]{};
\vertex[fill] (a4) at (0.25,-1) [label=below:$v_2$][]{};
\vertex[fill] (a5) at (0.75,-1){};
\vertex[fill] (a6) at (1,0) {};
\vertex[fill] (a7) at (1.25,0) {};
\vertex[fill] (a8) at (1.5,0) {};
\vertex[fill] (a9) at (1.75,0) {};
\vertex[fill] (a10) at (2,0) {};
\vertex[fill] (a11) at (2,-1) {};
\vertex[fill] (a12) at (0.5,-1) {};
\vertex[fill] (a13) at (2.25,0) {};
\path [line width=1pt]
(a1) edge (a4)
(a1) edge (a3)
(a11) edge (a9)
(a5) edge (a6)
(a11) edge (a8);
\end{tikzpicture}
\caption{A drawing of $ P^{1} $ and $P^{2}$ with assumption that $u$ of $P^{1}$ has two neighbors $v_1,v_2$ in $P^{2}$.}
\label{f4}
\end{figure}

\begin{figure}
\centering
\begin{tikzpicture}
\tikzstyle{vertex2}=[circle, draw, inner sep=0pt, minimum size=4pt]
\draw[line width=1pt] (0,0) -- (2.25,0);
\draw[line width=1pt] (-1,-1) -- (1.75,-1);
\vertex[fill] (a1) at (0,0) [label=above:$u$][label=left:$P^{1}:\quad\quad\;\quad\;$]{};
\vertex[fill] (a2) at (-1,-1) [label=below:$y$][label=left:$P^{2}:\:\:\:$]{};
\node[vertex2,fill] (a3) at (0,-1) [label=below:$\mathbf{v}$] {};
\vertex[fill] (a5) at (1,-1){};
\vertex[fill] (a6) at (1,0) {};
\vertex[fill] (a7) at (1.25,0) {};
\node[vertex2,fill] (a9) at (1.75,0) {};
\vertex[fill] (a10) at (2,0) {};
\vertex[fill] (a11) at (1.75,-1) {};
\vertex[fill] (a12) at (0.5,-1) {};
\vertex[fill] (a13) at (2.25,0) {};
\path [line width=1pt]
(a5) edge (a6)[line width=1pt];
\path [line width=1.5pt]
(a11) edge (a9)
(a1) edge (a3);
\end{tikzpicture}
\caption{A ladder drawing of the chains $ P^{1}$ and $P^{2}$  shown in Figure~\ref{f4}.}
\label{f5}
\end{figure}
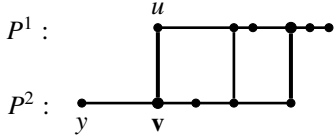
Also, let $a_1,\ldots,a_m\in P^{1}$ and  
$b_1,\ldots,b_m\in P^{2}$ be the vertices in a ladder drawing of $P^{1}$ and $P^{2}$ such that $a_ib_i$ are 
the segments of this drawing. Note that $a_ib_i$ is drawn as a vertical segment for each $i\in[m]$.
The section $S_i$, where $i \in [m-1]$, corresponding to this ladder drawing, is defined as induced subgraph on 
the vertex set $V_i=\{ v \in V(P^{1}) \cup V(P^{2}) \colon a_i \leq_{P^{1}} v \leq_{P^{1}} a_{i+1} \; or \;  b_i \leq_{P^{2}} v \leq_{P^{2}} b_{i+1} \}$. 
If there is at least one vertex preceding (resp. succeeding) $a_1$ or $b_1$ (resp. $a_m$ or $b_m$) then $S_0$ (resp. $S_m$) is 
defined as induced subgraph on the vertex set $V_0=\{ v \in V(P^{1}) \cup V(P^{2}) \colon v \leq_{P^{1}} a_{1} \; or \;   v \leq_{P^{2}} b_{1} \}$ 
(resp. $V_m= \{ v \in V(P^{1}) \cup V(P^{2}) \colon a_{m} \leq_{P^{1}} v  \; or\;   b_{m} \leq_{P^{2}} v   \}$). 
Also each of the vertices $a_{i}$ and $b_{i}$, where $i \in [m]$, is called a {\it boundary vertex}.

\begin{lemma}\label{pro}
Any graph $G$ satisfying the aforementioned properties has a standard drawing whose 
parallel paths are $P^{1},P^{2},P^{3}$.
\end{lemma}
\begin{proof}{ 
Consider a ladder drawing of  $P^{1}$ and $P^{2}$. 
We claim that we can draw $P^{3}$ as a horizontal path located above $P^{1}$ and the 
edges between $V(P^{3})$ and $V(P^{1})\cup V(P^{2})$ as segments such that no two segments have interruption. 
We proceed the proof of this claim by induction on $n=|V(P^{3})|\geq 1$. 
Let $n=1$, i.e., $P^{3}=z$.
The difficult case is the case that $\deg(z)=2$ (the cases $\deg(z)=0,1$ are trivial). 
Let $N(z)=\{u,v\}$. 
By Properties~5 and 6, there is a unique section $S_i$ such that $u,v\in V(S_i)$.  
By putting the vertex $z$ above the section $S_i$ (not above its boundary vertices), 
it is clear that we can draw the two segments $zu$ and $zv$ such these two segments have no interruption with other  
segments proving the claim for $n=1$. 
Now, let $n\geq 2$ and $P^{3}=z_1z_2\ldots z_n$. 
Without loss of generality, we may assume that each $z_i$ has at least one neighbor in $V(P^{1})\cup V(P^{2})$. 
Delete the vertex $z_n$ from $P^{3}$ to obtain the path $P'^{3}$.
By the induction, we can draw $P'^{3}$ as a horizontal path located above $P^{1}$ and the 
edges between $V(P'^{3})$ and $V(P^{1})\cup V(P^{2})$ as segments such that no two segments have interruption. Let $x'$ and $y'$ be respectively the last vertices of $P^{1}$ and $P^{2}$ 
having some neighbor in $P'^{3}$. (By Property~3, $|N(z_{n-1})\cap \{x',y'\}|\geq1$.) Because of similarity, assume that $y'z_{n-1}, x'z_j\in E(G)$.
Let $S_k$ be a section containing the neighbors of $z_n$. (Note that by properties~3,~4 and~6 there is no vertex in $P'^{3}$ having some neighbors in $S_j$, where $k<j$). Move all the vertices located inner the section
$S_k$ and after $x'$ to the right to pass the intersection of segment $yz_{n-1}$ and $P^{1}$. 
Now, stretch the section $S_k$ (we move $a_{k+1}$, $b_{k+1}$ and all vertices after these two vertices 
(if there exist and it is necessary) to the right side so by adding $z_n$ to $P'^{3}$, the edges adjacent to $z_n$ can be drawn in a 
way causing no segment interruption). 
The other case that  $x'z_{n-1}, y'z_j\in E(G)$ can be settled similarly. 
Therefore, the claim is proved via induction. 
To finish the proof of lemma, it suffices to split the the thick edges and vertices in the drawing whose existence is insured by the claim, 
so that no interruption is made, see Figure~\ref{f8}.
}\end{proof}


\begin{figure}
\centering
$$\begin{array}{ccc}
\begin{tikzpicture}
\tikzstyle{vertex2}=[circle, draw, inner sep=0pt, minimum size=4pt]
\draw[line width=1pt] (-1,0) -- (2.25,0);
\draw[line width=1pt] (0,-1) -- (1.75,-1);
\draw[line width=1pt] (-1,1) -- (2.25,1);
\node[vertex2,fill] (a1) at (0,0) [label=left:$P^{1}:\quad\quad\:\quad\;$]{};
\vertex[fill] (a2) at (1.5,-1) {};
\vertex[fill] (a3) at (0,-1)[label=left:$P^{2}:\quad\:\:\quad\quad$]{};
\vertex[fill] (a5) at (1,-1){};
\vertex[fill] (a6) at (1,0) {};
\vertex[fill] (a7) at (1.25,0) {};
\node[vertex2,fill] (a9) at (1.75,0) {};
\vertex[fill] (a10) at (2,0) {};
\vertex[fill] (a11) at (1.75,-1) {};
\vertex[fill] (a12) at (0.5,-1) {};
\vertex[fill] (a13) at (2.25,0) {};
\vertex[fill] (a14) at (-0.25,1) [label=left:$P^{3}:\quad\:\:\;\quad$]{};
\vertex[fill] (a15) at (-0.5,1) {};
\vertex[fill] (a16) at (-0.25,0) {};
\vertex[fill] (a17) at (1.75,1) {};
\vertex[fill] (a18) at (1.5,1) {};
\vertex[fill] (a19) at (1.90,1) {};
\path [line width=1pt]
(a14) edge (a12)[line width=1pt]
(a18) edge (a7)[line width=1pt]
(a15) edge (a16)[line width=1pt]
(a17) edge (a2)[line width=1pt]
(a19) edge (a10)[line width=1pt]
(a5) edge (a6)[line width=1pt];
\path [line width=1.5pt]
(a11) edge (a9)
(a1) edge (a3);
\end{tikzpicture}

&\qquad\quad &
\begin{tikzpicture}
\draw[line width=1pt] (-1,0) -- (2.25,0);
\draw[line width=1pt] (0,-1) -- (1.75,-1);
\draw[line width=1pt] (-1,1) -- (2.25,1);
\vertex[fill] (a1) at (0,0) [label=left:$P^{1}:\quad\quad\:\quad\;$]{};
\vertex[fill] (a2) at (1.5,-1) {};
\vertex[fill] (a3) at (0,-1)[label=left:$P^{2}:\quad\:\:\quad\quad$]  {};
\vertex[fill] (a5) at (1,-1){};
\vertex[fill] (a6) at (1,0) {};
\vertex[fill] (a7) at (1.25,0) {};
\vertex[fill] (a9) at (1.75,0) {};
\vertex[fill] (a21) at (1.85,0) {};
\vertex[fill] (a20) at (-0.15,0) {};
\vertex[fill] (a10) at (2,0) {};
\vertex[fill] (a11) at (1.75,-1) {};
\vertex[fill] (a12) at (0.5,-1) {};
\vertex[fill] (a14) at (-0.25,1) [label=left:$ P^{3}:\quad\:\:\;\quad$]{};
\vertex[fill] (a15) at (-0.5,1) {};
\vertex[fill] (a16) at (-0.25,0) {};
\vertex[fill] (a17) at (1.75,1) {};
\vertex[fill] (a18) at (1.5,1) {};
\vertex[fill] (a19) at (1.90,1) {};
\path [line width=1pt]
(a14) edge (a12)[line width=1pt]
(a18) edge (a7)[line width=1pt]
(a15) edge (a16)[line width=1pt]
(a17) edge (a2)[line width=1pt]
(a19) edge (a10)[line width=1pt]
(a9) edge (a11)[line width=1pt]
(a1) edge (a3)[line width=1pt]
(a20) edge (a3)[line width=1pt]
(a18) edge (a19)[line width=1pt]
(a11) edge (a21)[line width=1pt]
(a5) edge (a6)[line width=1pt];

\end{tikzpicture}

\end{array}$$
\caption{The picture on the right is derived from the left picture by splitting any thick edge and vertex into two edges and two vertices so that no interruption is made.}
\label{f8}
\end{figure}
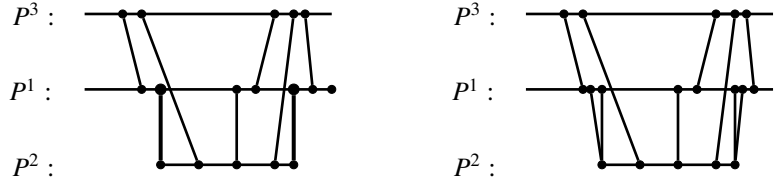

 
Now, we are in a position to prove the following lemma which somehow unifies all the lemmas in this section.

\begin{lemma}\label{representation}
Let $G$ be a graph with $\Delta(G)\leq 3$ and $ F(G)=3 $. 
There are an $F(G)$-set, a chain set $\S=\{R_1,R_2,R_3\}$ corresponding to it,
and a standard drawing of $G$ whose 
parallel paths are $R_1,R_2,R_3$. 
\end{lemma}
\begin{proof}{
Let $F=\{x,y,z\}$ be an $F(G)$-set and $\S=\{R_1,R_2,R_3\}$ a  chain set corresponding 
to it whose existence are ensured by Lemma~\ref{unless} such that $\S$ has 
the minimum possible number of non-trivial chains.  We also assume that $x,y,$ and $z$ are respectively 
the initial vertices of chains $R_1,R_2,$ and $R_3$. In what follows, we shall prove that $G$ has a standard drawing whose 
parallel paths are $R_1,R_2,R_3$, so the assertion is proved. 

Note that when we have only one non-trivial chain in $\S$, then 
by drawing the non-trivial chain as a horizontal line, putting the two trivial chains on the different sides of it, 
and drawing the edges as segments, clearly the edges between the chains do not interrupt each other and 
we have a standard drawing of $G$. If the chains in $\S$ are non-trivial then,  
by Lemmas~\ref{cross},~\ref{crosss}, and~\ref{unless}, 
the graph $G$ with paths $R_1,R_2$ and $R_3$ satisfies the conditions stated before Lemma~\ref{pro}. 
Therefore by Lemma~\ref{pro}, it is clear that $G$ has a standard drawing whose parallel paths are $R_1,R_2,R_3$.
In the rest of the proof, we assume that only two chains  in $\S$ are non-trivial.\\

For simplicity, we may assume that the chains $R_1$ and $R_2$ are non-trivial and $R_3$ is trivial. 
Consider a ladder drawing of  $R_1$ and $R_2$, where $R_1$ is located above $R_2$. Note that the chain $R_3$ is just a vertex $z$. 
Without loss of generality, we may assume that the number of neighbors of $z$ in $R_1$ is~not less than those in $R_2$.  
Now, we put $z$ above $R_1$ such that the edge between $z$ and its neighbor in $R_2$ (if exists) is a vertical segment, 
see Figure~\ref{f6}. Now, we can draw two the other edges adjacent to $z$ as straight segments.
To obtain the desired drawing, it suffices to split any thick edge and vertex into two edges and two vertices (backking to the original situation)
so that no interruption is made, see Figure~\ref{f7}. So we proved the assertion.\hfill\\

\begin{figure}
\centering
\begin{tikzpicture}
\tikzstyle{vertex2}=[circle, draw, inner sep=0pt, minimum size=4pt]
\draw[line width=1pt] (0,0) -- (2.25,0);
\draw[line width=1pt] (-1,-1) -- (1.75,-1);
\vertex[fill] (a1) at (0,0) [label=above:$x$][label=left:$R_{1}:\quad\quad\;\quad\;$]{};
\vertex[fill] (a2) at (-1,-1) [label=below:$y$][label=left:$R_{2}:\:\:\:$]{};
\node[vertex2,fill] (a3) at (0,-1) [label=below:$\mathbf{v}$] {};
\vertex[fill] (a5) at (1,-1){};
\vertex[fill] (a6) at (1,0) {};
\vertex[fill] (a7) at (1.25,0) {};
\node[vertex2,fill] (a9) at (1.75,0) {};
\vertex[fill] (a10) at (1.5,0) {};
\vertex[fill] (a11) at (1.75,-1) {};
\vertex[fill] (a12) at (-0.5,-1) {};
\vertex[fill] (a13) at (2.25,0) {};
\vertex[fill] (a14) at (-0.5,1) [label=above:$z$][label=left:$R_{3}:\quad\:\quad$]{};
\path [line width=1pt]
(a14) edge (a12)[line width=1pt]
(a14) edge (a7)[line width=1pt]
(a14) edge (a10)[line width=1pt]
(a5) edge (a6)[line width=1pt];
\path [line width=1.5pt]
(a11) edge (a9)
(a1) edge (a3);
\end{tikzpicture}
\caption{A ladder drawing of $ R_1,R_2 $ with trivial chain $R_3:(z)$ above $R_1$ such that the edge between $z$ and its neighbor in $R_2$ (if exists) is a vertical segment.}
\label{f6}
\end{figure}
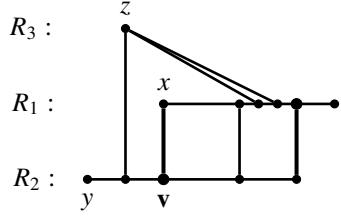

\begin{figure}
\centering
\begin{tikzpicture}
\draw[line width=1pt] (0,0) -- (2.25,0);
\draw[line width=1pt] (-1,-1) -- (1.75,-1);
\vertex[fill] (a1) at (0,0) [label=above:$x$][label=left:$R_{1}:\quad\quad\;\quad\;$]{};
\vertex[fill] (a2) at (-1,-1) [label=below:$y$][label=left:$R_{2}:\:\:\:$]{};
\vertex[fill] (a3) at (-0.25,-1) [label=below:$v_1$][]{};
\vertex[fill] (a5) at (1,-1){};
\vertex[fill] (a6) at (1,0) {};
\vertex[fill] (a7) at (1.25,0) {};
\vertex[fill] (a9) at (1.65,0) {};
\vertex[fill] (a10) at (1.5,0) {};
\vertex[fill] (a11) at (1.75,-1) {};
\vertex[fill] (a12) at (-0.5,-1) {};
\vertex[fill] (a13) at (2.25,0) {};
\vertex[fill] (a14) at (-0.5,1) [label=above:$z$][label=left:$R_{3}:\quad\:\quad$]{};
\vertex[fill] (a15) at (0.25,-1) [label=below:$v_2$] {};
\vertex[fill] (a16) at (1.80,0) {};
\path [line width=1pt]
(a1) edge (a3)
(a1) edge (a15)
(a11) edge (a16)
(a11) edge (a9)
(a14) edge (a12)
(a14) edge (a7)
(a14) edge (a10)
(a5) edge (a6);
\end{tikzpicture}
\caption{A drawing of $ R_1,R_2 $ and $R_3$ shown in Figure~\ref{f6} such that the edges between them do not interrupt each other.}
\label{f7}
\end{figure}
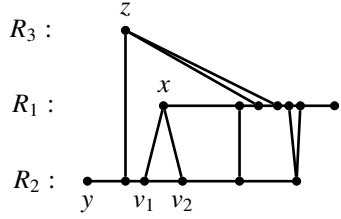
}\end{proof}
Note that this lemma is stronger than Observation~\ref{p1} and Theorem~\ref{ftwo}. 
\section{Proof of Main Results}
This section is mainly concerned with proving the main results stated in Section~\ref{mainresults}. 
In the following proposition, we present an upper bound for the forcing number of a graph of $3$-parallel paths. 
\begin{proposition} \label{left} 
For any graph $ G $ of $3$-parallel paths, $ F(G) \leq 3$. In particular, 
the left-most vertices of any standard drawing of $G$ form a forcing set. 
\end{proposition}
\begin{proof}{
Consider a standard drawing of a graph $G$ with parallel paths $Q_{1}, Q_{2}, Q_{3}$, 
where their indices are set according to their position in this drawing, i.e.,   for $i<j$, the path $Q_i$ is located above the path $Q_j$. 
Also, let $x,y,z$ be respectively the left-most vertices of $Q_{1}, Q_{2}, Q_{3}$. 
We claim that $F=\{x,y,z\}$ is a forcing set of $G$ which completes the proof.
We proceed the proof of this claim by induction on $|V(G)|+|E(G)|$. For $|V(G)|+|E(G)|=3$, we clearly have the assertion (in this case, $G$ has no edge). 
We assume that $|V(G)|+|E(G)|\geq 4$. If $Q_i$ is a path with one vertex for some $i\in[3]$, then it is~not difficult to see that $F$ is a forcing set of $G$. 
Therefore, we may suppose that each of $Q_1$, $Q_2$, and $Q_3$ has at least two vertices. 
If $\deg(v)=1$ for some $v\in\{x,y,z\}$, then, by induction, $F'=(F\setminus\{v\})\cup\{u\}$, where $u$ is the unique neighbor of $v$, 
is a forcing set of $G-v$ which implies that $F$ is a forcing set of $G$ as well. 
Henceforth, we assume that  $\deg(v)\geq 2$ for each $v\in\{x,y,z\}$. 
Now, we consider the case that $G[F]$, the induced subgraph on $F$, 
has at least edge $e$. It is clear that removing this edge cannot affect change the situation, i.e, $F$ is a forcing set of $G$ if and only if $F$ is a forcing set of $G-e$
But, in view of the induction, we know that $F$ is a forcing set of $G-e$. 
For a contradiction, assume that each of $x,y,$ and $z$ has the degree at least two and no two of them are adjacent.  
It is easy to see that either $x$ has a neighbor on $Q_3$ or $z$ has a neighbor on $Q_1$. 
By similarity, we may assume that $x$ has a neighbor $z'$ on $Q_3$. Note that $z$ is located in $Q_3$ at the left side of $z'$. 
Consider $Q_{1}, Q_{2}, Q_{3}$ as three parallel unbounded  lines in $\mathbb{R}^2$. Let $a$ be the intersection of the segment $xz'$ and $Q_2$. 
Clearly, $z$ has no neighbor on $Q_1$, otherwise, we have an interruption. 
Therefore, $z$ has a neighbor $y'$ in $Q_2$  located in the left side of $a$. Moreover, $y$ is in $Q_2$ in the left side of $y'$. 
Now, it is clear that $y$ has a neighbor neither in $Q_2$ nor $Q_1$, a contradiction. 
}\end{proof}
\begin{remark}\label{rem1}
With a similar approach, we can prove that for graphs of $k$-parallel paths, the zero forcing number is at most $k$, which is a result in favor of Conjecture~\ref{conj}.  
\end{remark}
\subsection{Proof of Theorem~\ref{iffkparallel}}\label{pt2} 
The previous proposition implies that a graph of $3$-parallel paths has a forcing set of size at most three. 
By the succeeding statement, we show that the forcing number of these graphs is three.  
\begingroup
\def\thetheorem{\ref{iffkparallel}}
\begin{theorem}
For a graph $ G $ with $ \Delta(G) \leq 3 $, $ F(G)=3 $ if and only if $ G $ is a graph of $3$-parallel paths.
In particular, the left-most vertices of parallel paths in any standard drawing of a graph of $3$-parallel paths form an $F(G)$-set. 
\end{theorem}
\addtocounter{theorem}{-1}
\endgroup
\begin{proof}{
Let $ G $ be a graph of $3$-parallel paths with $ \Delta(G) \leq 3 $. By Proposition \ref{left}, we know $ F(G) \leq 3 $.
If $ F(G) = k < 3 $ then,  by Observation \ref{p1} and Theorem \ref{ftwo}, $ G $ is either a path or a graph of $2$-parallel paths which is impossible. 
Conversely, let $ G $ be a graph with  $ F(G)=3 $ and $ \Delta(G) \leq 3 $. By Lemma \ref{representation}, there is a standard drawing of $G$ 
whose parallel paths are the chains of a chain set corresponding to an $ F(G) $-set. So $G$ is a graph of $k$-parallel paths, 
where $ k \in [3] $. If $G$ is a graph of $k$-parallel paths for some 
$k < 3$ then, by Observation~\ref{p1} and Theorem~\ref{ftwo}, $ F(G) = k < 3 $ which is a contradiction.  
Therefore, $G$ is a graph of $3$-parallel paths. Now, by 
Proposition~\ref{left}, the left-most vertices of the parallel paths in any standard drawing of $G$ form an $F(G)$-set.
}\end{proof}
Although we use Observation \ref{p1} and Theorem \ref{ftwo} to prove that $F(G)\neq 1,2$, 
we can avoid of using these observation and theorem and use Lemma~\ref{representation}. 

\subsection{Proof of Theorem~\ref{fmk}}\label{pt3}
In this section, we present a proof of Theorem~\ref{fmk}. Recall its statement.

\begingroup
\def\thetheorem{\ref{fmk}}
\begin{theorem}
For a graph $G$ with $\Delta(G) \leq 3$, the following assertions hold. 
\begin{enumerate}
\item $F(G)=M(G)=1$ if and only if $G$ is a path. 
\item $F(G)=M(G)=2$ if and only if $G$ is a graph of $2$-parallel paths. 
\item $F(G)=M(G)+1=3$ if and only if $G$ is a graph of type defined in Figure~\ref{F8}. 
\item $F(G)=M(G)=3$ if and only if $G$ is a graph of $3$-parallel paths and not of type defined in Figure~\ref{F8}.
\item There is no graph with $F(G)=M(G)+1=2$ or $F(G)=M(G)+2=3$.
\end{enumerate}
\end{theorem}
\addtocounter{theorem}{-1}
\endgroup
\begin{proof}{
The first assertion is an immediate consequence of Proposition~\ref{p1} and Theorem~\ref{mf}. 
Using Theorems~\ref{mtwo} and~\ref{ftwo}, we would have the second assertion. 
To prove the third item, assume that $F(G)=M(G)+1=3$. Since $F(G)=3$, by Theorem~\ref{iffkparallel}, $G$ is a graph of $3$-parallel paths. 
Therefore, in view of Theorem~\ref{mtwo},  $G$ is a graph in $\mathcal{F}$. The family $\mathcal{F}$ contains $6$ types of graphs.
Since $\Delta(G)\leq 3$, the only possible type is the one described in Figure~\ref{F8}. 
Conversely, if $G$ is of the type defined in Figure~\ref{F8}, then $G$ is in  $\mathcal{F}$ and, by Theorem~\ref{mtwo},  $M(G)=2$.
On the other hand, $G$ is a graph of $3$-parallel paths and so, by Theorem~\ref{iffkparallel}, $F(G)=3$, completing the proof of third assertion.  

As an immediate consequence of Theorem~\ref{iffkparallel}, $F(G)=M(G)=3$ is sufficient for being $3$-parallel paths. 
Also, by Theorem~\ref{mtwo}, $G$ is~not of type defined in Figure~\ref{F8}. 
Conversely, if $G$ is a graph of $3$-parallel paths and not of types defined in Figure~\ref{F8}, then $1\leq M(G)\leq F(G)\leq 3$ by Propositions~\ref{mg} and~\ref{left};  
and $M(G)\neq 1,2$ using Theorem~\ref{mtwo} and Theorem~\ref{M=0} completing the proof of forth item. 

To prove the last item, note that when $F(G)=2$, $G$ is a graph of  $2$-parallel paths (Theorem~\ref{ftwo}) and hence, by Theorem~\ref{mtwo}, $M(G)=2$. 
Also, when $F(G)=3$, $G$ is a graph of $3$-parallel paths (Theorem~\ref{iffkparallel}). But, we know that $M(G)=1$ if and only if $G$ is a path, see Theorem~\ref{M=0}.
}\end{proof}

\subsection{Proof of Corollary~\ref{ftn}}\label{pt4} 
In this section, we present the proof of Theorem 6. Recall its statement.

\begingroup
\def\thetheorem{\ref{ftn}}
\begin{corollary}
For a graph $G$ of $k$-parallel paths and without isolated vertices, $F_{t}(G) \leq 2k$ and this bound is sharp.
\end{corollary}
\addtocounter{theorem}{-1}
\endgroup 
\begin{proof}{ 
Consider a standard drawing of $G$. 
By Theorem \ref{iffkparallel}, the left-most vertices of the parallel paths form a forcing set. It is clear that the left-most two vertices of the 
parallel paths form an $TF$-set. In the parallel paths isomorphic to $P_1$, 
we consider the only vertex of these paths with one of the vertices adjacent to it. 
So $F_{t}(G) \leq 2k$.
Now we show this bound is sharp. Let $ G $ be a graph which is the union of $k$ vertex disjoint paths so that the size of paths is at least two. 
It is clear that $G$ is a graph of $k$-parallel paths and any $TF$-set must contain at least two vertices from each path, which implies $F_{t}(G) \geq 2k$. 
As observed earlier $ F_{t}(G) \leq 2k $.  So, $ F_{t}(G) = 2k $.
}\end{proof} 

It is known that for a graph $G$ with no isolated vertex (see~\cite{db,dho}), $F_t(G)/F(G) \leq 2$. 
The previous corollary concludes that this bound is sharp. To see this, set $G$ be a graph of $k$ independent 
induced paths so that the size of paths is at least two and whose vertices are not adjacent to each other. It is clear that 
$G$ is a graph of $ k $- parallel paths, then $ F(G)=k $ and $ F_{t}(G)=2k $.


\end{document}